\documentclass[12pt,a4paper,reqno]{amsart}
\usepackage[english]{babel}
\usepackage[applemac]{inputenc}
\usepackage[T1]{fontenc}
\usepackage{palatino}
\usepackage{amsmath}
\usepackage{amssymb}
\usepackage{amsthm}
\usepackage{amsfonts}
\usepackage{graphicx}
\usepackage{color}
\usepackage{vmargin}
\usepackage{mathtools}
\usepackage{pgf,tikz}
\usepackage{subfig}
\usepackage{pgfplots}
\usepackage{mathrsfs}
\usetikzlibrary{arrows}
\setmarginsrb{20mm}{20mm}{20mm}{20mm}{10mm}{10mm}{10mm}{10mm}

\theoremstyle{plain}
\newtheorem{thm}{Theorem}[section]

\newtheorem{lemma}[thm]{Lemma}

\theoremstyle{definition}
\newtheorem{definition}{Definition}[section]
\newtheorem{example}{Example}
\newtheorem{remark}{\textnormal{\textbf{Remark}}}[section]
\newtheorem*{remark*}{\textnormal{\textbf{Remark}}}

\newtheorem{claim}[thm]{Claim}

\newcommand*{\diam}{\mathrm{diam}}
\newcommand*{\dist}{\mathrm{dist}}

\newcommand{\R}{{\mathbb R}}

\newcommand{\Om}{\Omega}

\definecolor{blau}{rgb}{0.1,0.0,0.9}

\newcommand{\red}{\color{red}}
\newcounter{komcounter}
\numberwithin{komcounter}{section}

\newcommand{\kom}[1]{}
\renewcommand{\kom}[1]{$\backslash$ {\bf  \red #1}$\backslash$}

\title[]{$\varepsilon$-approximability and Quantitative Fatou Theorem on Riemannian manifolds}
\author{Marcin Grysz\'owka{\small$^*$}}\thanks{{\small$^*$}\,The author was supported by the National Science Center, Poland (NCN), UMO-2020/39/O/ST1/00058}
\address{Institute of Mathematics, Polish Academy of Sciences,
ul. \'Sniadeckich 8, 00-656 Warsaw, Poland\/ \and Faculty of Mathematics, Informatics and Mechanics, University of Warsaw, ul. Banacha 2, 02-097 Warsaw, Poland}
\email{mgryszowka@impan.pl}

\begin{document}

\begin{abstract}
We prove the Quantitative Fatou Theorem for Lipschitz domains on complete Riemannian manifolds. This requires extending the $\epsilon$-approximation lemma to the manifold setting. Our studies apply to harmonic functions, as well as to a class of A-harmonic functions on manifolds.
The presented results extend works by Dahlberg, Garnett, Bortz and Hofmann.
\end{abstract}

\maketitle

\section{Introduction}

Classical Fatou theorem states that bounded harmonic functions have non-tangential limits at a.e. point of the boundary. Garnett proved its quantitative version for upper half plane, see \cite[Chapter 8.6]{garnett}. Let us briefly explain what we mean by quantitative version. Consider a counting function $N_{\varepsilon}(x)$, where $x$ is a point on the boundary, which counts oscillations of size at least $\varepsilon$ of bounded harmonic function $u$ in a nontangential cone with vertex at $x$ and aperture $\alpha$. Moreover, let us take only lacunary sequences approaching to the boundary, with coefficient given by $\theta$. The value of a counting function is then the supremum over all possible lengths of such sequences. Garnett proved that for $\|u\|_{\infty}\le 1$ this function satisfies the estimate
\begin{equation*}
\int_{I}N_{\varepsilon}(x)dx\le C\varepsilon^{-7}
\end{equation*}
for every $\varepsilon>0$ with constant depending on $\varepsilon, \alpha, \theta$, but independent of $u$, where $I$ is any interval of length $1$. 
The crucial part of the proof was $\varepsilon$-approximability. The property was first established by Varopoulos \cite{var}, then refined by Garnett \cite{garnett}. Later Dahlberg proved that $\varepsilon$-approximability holds for harmonic functions on Lipschitz domains \cite{dahlberg}, then in \cite{kkpt} it was proved for A-harmonic functions, i.e. solutions to real divergence form equation ${\rm div}A\nabla u=0$.
Fairly recently, it was shown that for real divergence form operator satisfying the Carleson measure condition and pointwise local Lipschitz bound, $\varepsilon$-approximability is equivalent to uniform rectifiability, see \cite{hmm}, \cite{azzam}.

It was then proven by Bortz and Hofmann in \cite{bortz-hoffman} that Quantitative Fatou Property is equivalent to uniform rectifiability in the Euclidean setting (under appropriate assumptions). Therefore, there is a deep connection between properties of PDEs and geometry of sets on which they are defined.

So far we have not encountered any work regarding Quantitative Fatou Property in a setting different than Euclidean. Thus we would like to extend previous results to the setting of Riemannian manifolds. They are much broader than Euclidean space, however in a sense thay are a first step in generalizing any results from $\mathbb{R}^n$ to more general metric measure spaces, \cite{mt}. Since some notions are not even yet defined outside Euclidean space or fully established,  e.g. uniform rectifiability, see \cite{bate}, \cite[Chapters 6 and 9]{mattila}, we deal with the case of Lipschitz domains. We prove the following:
\begin{thm}\label{thm1}
Let $M$ be a complete Riemannian manifold and let further $\Omega\subset M^n$ be a Lipschitz domain. Furthermore, let $u:\Omega\rightarrow\mathbb{R}$ be a harmonic bounded function with $\|u\|_{\infty}\le 1$. Then for every point $p\in\partial\Omega$

\begin{align*}
\sup_{\substack{0<r<{\rm r_{inj}}}}\frac{1}{r^{n-1}}\int_{\partial\Omega\cap B(p,r)}N(r,\varepsilon,\theta)(q)d\sigma(q)\le C(\varepsilon,\alpha,\theta,n,\Omega),
\end{align*}

where $\varepsilon,\alpha,\theta$ are constants in the definition of the counting function. In particular, constant $C$ is independent of $u$. 
\end{thm}

One of the key auxiliary results to prove Theorem \ref{thm1} is the following $\varepsilon$-approximability property:

\begin{thm}\label{thm-e-approx}
Let $M$ be an $n$-dimensional complete Riemannian manifold and $\Omega\subset M$ be an open bounded connected Lipschitz set. Let $u$ be a harmonic bounded function in $\Omega$. Then $u$ is $\varepsilon$-approximable for every $\varepsilon >0$. 
\end{thm}

Main difficulty is the fact that we do not have one map available on whole of $\Omega$. Therefore we deal with pieces of $\Omega$ where there are maps. On these pieces we take local $\varepsilon$-approximants. We need to show that we can choose all maps in a uniform way and that we can later glue everything together to obtain $\varepsilon$-approximability on $\Omega$. In Euclidean case there is usually a certain way to integrate counting function to obtain desired estimate. We deal with it by taking an appropriately constructed curve and its neighbourhood contained in a non-tangential cone. Using coarea formula will then enable us to arrive at our estimate.  In our proof we follow the idea of \cite{kkpt}, but adjust it to the setting of Riemannian manifolds.

\section{Acknowledgements}
I would like to thank my advisor, Tomasz Adamowicz, for all the support he offered my while preparing this article. I would like to thank Katrin F{\"a}ssler for helping me with a part of the proof of Theorem \ref{glowne twierdzenie}.
Moreover, I thank Tuomas Orponen and Xavier Tolsa for conversations during Summer School in Jyv{\"a}skyl{\"a}.

\section{Preliminiaries}\label{prelim}

Throughout this article we work in the setting of Riemannian manifold.

\begin{definition}
Let $M^n$ be a smooth $n$-dimensional manifold. We will say that $M^n$ is a Riemannian manifold if it is equipped with a scalar product $g_x:T_x M^n\times T_x M^n\to \mathbb{R}$ depending smoothly on the point $x\in M$. Usually $g$ is represented by a matrix which is also denoted by $g$. Its determinant is denoted by $\det g$ and its coefficients are written with lower indices as $g_{ij}$ for $i,j=1,\dots,n$. The inverse matrix, which exists at every point since $g$ is invertible as scalar product, is denoted by $g^{-1}$ and its coefficients are written with upper indices $g^{ij}$ for $i,j=1,\dots,n$.
In what follows we will write $M:=M^n$ to denote the Riemannian manifold when the dimension $n$ is fixed.
\end{definition}

When working with manifold one often uses local coordinates. In our case a certain choice of coordinates is convienient.

\begin{definition}
Let $M$ be a Riemannian manifold. For any point $p\in M$ and any neighbourhood of $p$ we introduce \emph{normal coordinates} in a following way. Let $\exp_{p}:T_{p}M\rightarrow M$ be a map such that $\exp_{p}(v)=\gamma_{p}(1)$,
where $\gamma_{p}$ is a unique geodesic satisfying $\gamma_{p}(0)=p$ and $\dot{\gamma}_{p}(0)=v$.
It is known that one can find such a neighbourhood of point $p$ that $\exp_{p}$ is a diffeomorphism, call it $U_p$. Since the tangent space $T_p M$ can be identified with space $\mathbb{R}^n$, say by isomorphism $T:T_p M\rightarrow\mathbb{R}^n$, the map $T\circ\exp_p^{-1}:M\supset U_p\rightarrow\mathbb{R}^n$ defines a local coordinate chart which we will call \emph{normal coordinates}. In the sequel with an abuse of notation we will omit the isomorphism between a tangent space and Euclidean space and use only $\exp_p$ or $\exp_p^{-1}$.
\end{definition}

We will mostly deal with Lipschitz domains. Therefore, it is necessary to recall what we mean by a Lipschitz set on manifold $M$.

\begin{definition}
Let $\Omega$ be an open connected subset of a manifold $M$ and let $p\in\partial\Omega$. We say that $\Omega$ is \emph{locally Lipschitz} at $p$ if there exist a neighbourhood $U$ of $p$ in $M$ and a local chart $f:U\rightarrow f(U)\subset\mathbb{R}^n$ such that $f(U\cap\Omega)$ is a Lipschitz set in $\mathbb{R}^n$. We say that $\Omega$ is a Lipschitz set if it is locally Lipschitz at every point of its boundary $\partial\Omega$.
\end{definition}

\begin{definition}
Let $\Omega\subset M$ be an open set. We say that it satisfies inetrior corkscrew condition if there exists a constant $c$ such that for every set $B(p,r)\cap\partial\Omega$ with $p\in\partial\Omega$ and $0<r<\diam (\partial\Omega)$ there is a ball $B(\widetilde{x},cr)\subset B(p,r)\cap\Omega$. We say that $\Omega$ satisfies exterior corkscrew condition if $M\setminus\Omega$ satisfies interior corkscrew condition.
\end{definition}

\begin{definition}
We say that set $E\subset M$ is \emph{Ahlfors-David regular}, of Hausdorff dimension $n-1$, if it is closed and there exists a constant $C<0$ such that
\begin{align*}
\frac{1}{C}r^{n-1}\le\sigma(B(p,r)\cap E)\le Cr^{n-1}, \hspace{5mm} \textnormal{ for all } p\in E \textnormal{ and } 0<r<\diam (E),
\end{align*}
where $\sigma$ denotes a surface measure on $E$.
\end{definition}

The following generalizes Definition 2.4 in \cite{bortz-hoffman} to the setting of Riemannian manifolds. We retrieve that definition for $M=\mathbb{R}^n$.

\begin{definition}\label{A-harmonic-def}
We will say that a Sobolev function $u:\Omega\rightarrow\mathbb{R}$ is \emph{$A$-harmonic} if it satisfies the following equation
\begin{equation}\label{A-harmonic}
\textnormal{div}(A\nabla u)=0,
\end{equation}
understood in the weak sense, where $A:TM\rightarrow TM$ is such that for each point $x\in\Omega$ we have $A(x):T_x M\rightarrow T_x M$ is a linear map and for every fixed vector field $X$ the map $A(\cdot,X)$ is measurable. Furthermore, we assume that there is a constant $C>1$ such that
\begin{align*}
C^{-1}|\xi|^2\le g_x(A(x)\xi,\xi), \hspace{5mm} \|A\|_{\infty}\le C,
\end{align*}
for all $x\in\Omega$ and $\xi\in T_x M$.
Moreover, we impose that operator $A$ has bounded coefficients and satisfies the following conditions (2.5) and (2.6) in \cite{bortz-hoffman}. 
\begin{itemize}
\item[(2.5)] the Carleson measure condition
$$
    \sup_{x\in\partial \Omega,0<r<\diam(\partial\Omega)}\frac{1}{H^{n-1}(B(x,r)\cap\partial\Omega)}\int_{B(x,r)\cap\Omega}|\nabla A(X)|dX\le C<\infty,
    $$
\item[(2.6)] the pointwise gradient estimate
 $$
    |\nabla A(X)|\le \frac{C}{\dist(X,\partial\Omega)}, \hspace{5mm} \textnormal{for all } X\in\Omega.
   $$
\end{itemize}
\end{definition}

\begin{remark}
If in the above definition operator $A$ is such that for every $x\in\Omega$ we have $A(x):T_x M\rightarrow T_x M$ is an identity transformation, then we obtain a Laplace-Beltrami equation. For more information about harmonic functions on manifolds, see \cite{li}.
\end{remark}

\begin{definition}
Let $\Omega\subset M$ be an open set and $u\in L^1(\Omega)$. We say that $u$ has \emph{bounded variation} in $\Omega$ (denoted $u\in BV(\Omega)$) if 
$$
\sup\{\int_{\Omega}u \textnormal{div}(\phi X): X\in\Gamma(\Omega), \phi\in C_c^{\infty}(\Omega), |\phi|\le 1\}<\infty,
$$ 
where

$\Gamma(\Omega)=\{X: X \textnormal{ is a smooth vector field on } \Omega \textnormal{ and } g(X(x),X(x))\le 1 \textnormal{ for every } x\in\Omega\}$
and $g$ denotes the metric on $M$. 
The above supremum is called a \emph{variation} of $u$.
\end{definition}
\begin{remark}
If $\Omega\subset M=\mathbb{R}^{n}$, then the above definition gives the usual definition of functions of bounded variation in $\Omega$.
\end{remark}

The fundamental notion that will enable us to prove the Quantitative Fatou Property is the $\varepsilon$-approximability. It has been used by numerous authors. It seems that the first to describe $\varepsilon$-approximability on upper half-plane was Varopoulos \cite{var}, the property was later improved by Garnett \cite{garnett}. Later the property was proved for Lipschitz domains by Dahlberg, see \cite{dahlberg}. In \cite{kkpt} the result is proven for A-harmonic functions in Lipschitz domains. However, to our best knowledge so far, it has only been used in Euclidean setting. Therefore, we give the definition in the setting of Riemannian manifolds.

\begin{definition}\label{approx-def}
Let $\Omega\subset M$ be a Lipschitz domain on a Riemannian manifold $M$. Let $u:\Omega\rightarrow\mathbb{R}$ be a harmonic function with $\|u\|_{\infty}\le 1$. We will say that function $u$ is \emph{$\varepsilon$-approximable} for some $\varepsilon>0$ if there exists a function $\phi\in BV(\Omega)$ such that
\begin{enumerate}
\item $|u(x)-\phi(x)|<\varepsilon$ \hspace{5mm} for every  $x\in\Omega$,
\item $\int_{B(p,r)\cap\Omega}|\nabla\phi|\le Cr^{n-1}$ \hspace{5mm} for every $p\in\partial\Omega$ and $0<r<\diam (\Omega)$.
\end{enumerate} 
\end{definition}

Now we will define a generalized cone, a notion that in the Euclidean setting corresponds to the notion of the cone.

\begin{definition}\label{cone-definition}
Let $\Omega$ be a connected, open subset of a Riemannian manifold $M$ and let $p\in\partial\Omega$. Let $0<\alpha<\infty$. The set
$$
\Gamma(p):=\Gamma_{\alpha}(p)=\{q\in\Omega: d(p,q)\le(1+\alpha)d(q,\partial\Omega)\},
$$
is called a \emph{generalized cone}. Moreover, we also define a \emph{truncated generalized cone} as follows
$$
\Gamma^{r}(p):=\Gamma(p)\cap B(p,r).
$$
\end{definition}
For the sake of simplicity we will use a name "cone" instead of generalized cone. In literature a name non-tangential approach region is also used, see \cite{kkpt}.

\begin{definition}
Let $\Omega\subset M$ be open, bounded, connected and $p\in\partial\Omega$. Let further $\Gamma(p)$ denote a cone at $p$. We define a \emph{doubly truncated cone}
$\Gamma_{r_1,r_{2}}(p)$ as follows
$$
\Gamma_{r_1,r_{2}}(p):=(\Gamma(p)\cap B(p,r_{1}))\setminus (\Gamma(p)\cap B(p,r_2)),
$$
where $0<r_1,r_2<\infty$.
\end{definition}

The following is one of the key definitions of our work.

\begin{definition}[Counting function]\label{counting-function-definition}
Let $\Gamma^r(p)$ be a cone for some point $p$ in the boundary of $\Omega$. Let $u$ be a harmonic function defined on $\Omega$. Denote by $d$ a Riemannian distance in manifold $M$. Fix $\varepsilon>0$, $0<\theta<1$ and $0<r<1$. We will say that a sequence of points $Q_n\in\Gamma^r(p)$ is $(r,\varepsilon,\theta,p)$-admissible for u if 
\begin{align*}
|u(Q_n)-u(Q_{n-1})|\ge\varepsilon,\\
d(Q_n,p)<\theta d(Q_{n-1},p).
\end{align*}
Set 
$$N(r, \varepsilon, \theta)(p)=\sup\{k: \textnormal{there exists an } (r, \varepsilon, \theta, p)\textnormal{-admissible sequence of length }k\}.$$
We will call $N$ a \emph{counting function}.
\end{definition}
The similar notions are known in $\mathbb{R}^n$, see \cite{kkpt}, \cite{bortz-hoffman}. They are a way to estimate how much a harmonic function oscillates while approaching the boundary. In mentioned works the Quantitative Fatou Property was proved, i.e. this function is integrable and its integral over a surface ball is bounded by a measure of that ball. This gives a generalization of a classical Fatou theorem which states that non-tangential limit exists at a.e. point of boundary. In the language of counting function it reads that counting function $N$ is finite a.e. Therefore, the name Quantitative Fatou Property indeed suits, as it gives a quantitative bound of a certain integral. 

\section{Harmonic and A-harmonic $\varepsilon$-approximability}
The crucial part of the proof of Quantitative Fatou Theorem is $\varepsilon$-approximability. In Definition \ref{approx-def} we see that function $\phi$ is close to harmonic function $u$, but has a property that its gradient gives Carleson measure, which may not necessarily by true for any harmonic function. This is the essential part of estimates needed to obtain Quantitative Fatou Theorem.

\begin{thm}\label{thm-e-approx}
Let $M$ be an $n$-dimensional complete Riemannian manifold and $\Omega\subset M$ be an open bounded connected Lipschitz set. Let $u$ be a harmonic bounded function in $\Omega$. Then $u$ is $\varepsilon$-approximable for every $\varepsilon >0$. 
\end{thm}

In order to prove Theorem \ref{thm-e-approx} we use \cite[Theorem 2.15]{bortz-hoffman}. That theorem states that when $M$ is a Euclidean space, the assertion of Theorem \ref{thm-e-approx} holds. We cover set $\Omega$ with finitely many open sets such that on each we are able to use normal coordinates. The properties of these coordinates and the fact that $\Omega$ is Lipschitz enable us to prove existence of $\varepsilon$-approximation on each of the sets from the covering. Finally we glue these approximations and show that it is an $\varepsilon$-approximation on $\Omega$.
We need normal coordinates on a manifold in order to express Laplace equation on a manifold in these coordinates as an A-harmonic equation. We prove that a Lipschitz set satisfies interior corkscrew condition, which is an essential property of a set $\Omega$. Due to Lipschitzness the boundary $\partial\Omega$ satisfies Ahlfors-David regularity condition, which is necessary for our theorem to hold.

Before proving Theorem \ref{thm-e-approx} let us discuss two auxiliary results.First we establish a relation between BV functions on manifolds and $\mathbb{R}^n$.

\begin{lemma}\label{BV-mfld}
Let $M$ be a Riemannian manifold. Furthermore, let $\phi: U\to \R$ be a BV function defined on a set $U\subset V\subset\mathbb{R}^n$, where $V$ is such a set that $\exp_p$ is a diffeomorphism on $V$ into $M$ for some point $p\in V$. Then, the composition $\phi\circ\exp_p^{-1}$ is a BV function on $\exp_p(U)\subset M$.
\end{lemma}
\begin{proof}
By the definition, a function $\phi\in BV (U)$ if 
\[
\sup\int_U\phi\,{\rm div}v<\infty,
\]
%bounded variation on set $Z$ if 
%\[\sup\int_{Z}f{\rm div}v<\infty,\]
where the supremum is taken over the set of all $C^{\infty}_{c}(U,\mathbb{R}^n)$ with $|v|\le 1$. We would like to find uniform estimate for the integral $\int_{\exp_p(U)}(\phi\circ\exp_p^{-1}){\rm div}v$, where $v\in C^{\infty}_c(\exp_p(U),\mathbb{R}^n)$. By the change of variables  formula we obtain 
\begin{align*}
\int_{\exp_p(U)}(\phi\circ\exp_p^{-1}){\rm div}v &=\int_U((\phi\circ\exp_p^{-1})\circ\exp_p) {\rm div} (v\circ \exp_p)|J\exp_p|\\
&\lesssim\int_U\phi{\rm div}(v\circ\exp_p).
\end{align*}
%The transformation used in change of variables is $F:U\rightarrow \exp_p(U)$ given by $F(x)=\exp_p(x)$.
The approximate inequality is the consequence of the fact that Jacobian of $\exp_p$ is bounded, since set $U\subset V$.
In the last integral, instead of $v$, we now have $v\circ\exp_p$, which may a priori change the set of functions over which supremum is taken. However, it turns out that any function $w\in C^{\infty}_c (U,\mathbb{R}^n)$, $|w|\le 1$ can be written as $v\circ\exp_p$ for some function $v\in C^{\infty}_c (\exp_p(U),\mathbb{R}^n)$, $|v|\le 1$. Indeed, this follows from writing $w=(w\circ\exp_p^{-1})\circ\exp_p$ and setting $v=w\circ\exp_p^{-1}$. This completes the proof of the lemma.
\end{proof}

Next lemma is a mathematical folklore, however since we did not find its explicit proof in the literature we decided to provide the statement of the following result, whose proof can be found in Appendix A.

\begin{lemma}\label{lem-Lip-cork}
A bounded Lipschitz set in the Euclidean space satisfies interior corkscrew condition.
\end{lemma}

We are now in a position to prove the main result of this section.

\begin{proof}[Proof of Theorem~\ref{thm-e-approx}] 

We begin with covering domain $\Omega$ with appropriately constructed Lipschitz sets, such that each of these sets is diffeomorphic to a subset of $\mathbb{R}^n$ and all the partial derivatives of such diffeomorphisms are uniformly bounded in $\Omega$. Let us describe how it can be achieved.

Let us cover $\bar{\Omega}$ with sets $V_{p}$ for $p\in\Omega$, such that each $V_{p}$ is a neighbourhood of $p$ with the property $\exp_{p}$ is a diffeomorphism on $\exp_{p}^{-1}(V_{p})$. 
Since $\bar{\Omega}$ is closed and bounded, it holds that $\bar{\Omega}$ is compact by completeness of $M$. Therefore, there exists a finite cover of $\bar{\Omega}$ by sets denoted by $V_{l}$ with $l\in\{1,2,\dots,m\}$, where $m$ is a number of sets in the cover. Moreover, by $p_l$ we denote points $p$ corresponding to sets $V_l$. 
\smallskip
\\
{\bf Claim:} Sets $V_{l}$ can be choosen to be Lipschitz.
\smallskip
\\
 In order to prove the claim, take any ball $B(p_{l},R_l)\subset M$, where $R_l={\rm r_{inj}}(p_l)$ is injectivity radius at $p_l\in \Omega$, for $l=1,\ldots, m$. Due to the fact that $\bar{\Omega}$ is compact,  it holds that
 \[
 \inf_{p\in\Omega}{\rm r_{inj}}(p)=\min_{p\in \Omega}{\rm r_{inj}(p)}:={\rm r_{inj}}>0.
 \]
Next, consider balls $B(p_{l},R_l-\delta)$ with $\delta>0$ small enough, say $\delta=\frac{1}{10}r_{inj}$. If it turns out that such smaller balls do not cover $\Omega$, then increase $m$ so that the new larger family of balls covers $\Omega$. Since $B(p_{l},R_l-\delta)\subset\subset B(p_{l},R_l)$ and the map $\exp^{-1}_{p_{l}}$ is a diffeomorphism on $B(p_{l},R_l)$, we have that the derivative of $\exp_{p_{l}}$ is bounded on each of $B(p_{l},R_l-\delta)$. Moreover, balls $B(p_l,R_l-\delta)$ are Lipschitz because they are images of balls $B(0,R_l-\delta)$ in Euclidean space under $\exp_{p_l}$. This holds since the latter balls are Lipschitz sets in $\mathbb{R}^n$ and $\exp_p$ is a Lipschitz map.

Since in our setting function $u$ is defined only on $\Omega$, we need to intersect balls $B(p_l,R_l-\delta)$ with $\Omega$. Unfortunately, sets $V_l^{\prime}:=B(p_l,R_l-\delta)\cap\Omega$ need not be Lipschitz or satisfy interior corkscrew condition. Thus, it is necessary to augment our covering. 

Notice, that as $\Omega$ is Lipschitz, sets $V_l^{\prime}$ have finitely many connected components. Since $l\le m$, the set of all connected components of these sets is finite as well. Hence we may further assume that sets $V_l^{\prime}$ are connected.

Each set $V_l^{\prime}$ is locally Lipschitz at almost every point of the boundary $\partial V_l^{\prime}$. The set of points where $\partial V_l^{\prime}$ may fail to be Lipschitz is an $(n-2)$-dimensional set denoted by $b_l:=\partial\Omega\cap\partial B(p_l,R_l-\delta)$. If it happens that the intersection of boundaries has dimension $n-1$, then a set $V_l^{\prime}$ is Lipschitz and hence we do not have to worry about such a situation.

For $n=2$  sets $b_l$ are sets of disjoint points and for $n\ge 3$ sets $b_l$ are connected. Take such $n$-dimensional neighbourhoods $A_l\supset b_l$ that sets $V_l^{\prime}\setminus A_l$ are Lipschitz. We can also choose $A_l$ so that $d(V_l^{\prime}\setminus A_l,b_l)<\frac{r_{inj}}{2}$. It can be done because each set $V_l^{\prime}$ is "almost Lipschitz" and therefore we only have to take $A_l$ small enough so that the part of its boundary inside $V_l^{\prime}$ is Lipschitz and does not form cusps with $\partial V_l^{\prime}$, which means that $(\partial A_l\cap V_l^{\prime})\cup(\partial V_l^{\prime}\setminus A_l)$ is locally Lipschitz. Since there are finitely many sets $V_l^{\prime}$ the Lipschitz constants of all sets are uniformly bounded. Put $\mathcal{B}_l:=A_l\cap V_l^{\prime}$ and set

\begin{equation*}
V_l^{''}:=
\begin{cases}
V_l^{'} \hspace{5mm} \textnormal{ if } V_l^{'} \textnormal{ is Lipschitz}\\
V_l^{'}\setminus\mathcal{B}_l \hspace{5mm} \textnormal{ if } V_l^{'} \textnormal{ is not Lipschitz}.
\end{cases}
\end{equation*}

Sets $V_l^{''}$ are Lipschitz but they need not cover whole set $\Omega$. Thus we need to deal with sets $\mathcal{B}_l$. 
Divide sets $\mathcal{B}_l$ into subsets $\mathcal{B}_{l,i}$ such that:
\begin{itemize}
\item $|\{i:\mathcal{B}_{l,i}\}|\lesssim\frac{\textnormal{Vol} (\mathcal{B}_l)}{{\rm r_{inj}}^n}\lesssim\frac{\textnormal{Vol} (\Omega)}{{\rm r_{inj}}^n}$ \hspace{5mm} for every $l$.
\item $\diam (\mathcal{B}_{l,i})<{\rm r_{inj}}$ \hspace{5mm} for every $i$ and every $l$,
\item $\bigcup_{i} \mathcal{B}_{l,i}=\mathcal{B}_l$ \hspace{5mm} for every $l$, \hspace{5mm} for every $l$.
\end{itemize}

To construct such a partition, it is sufficient to cover set $\overline{\mathcal{B}}_l$ with sets $B(\widetilde{p}_{l,i},\frac{r_{inj}}{2})\cap\overline{\mathcal{B}}_l$ for points $\widetilde{p}_{l,i}\in\mathcal{B}_l$ and take a finite subcover.

Take sets $C_{l,i}:=B(p_{l,i}, {\rm r_{inj}}-\delta)\cap\Omega$ such that $\mathcal{B}_{l,i}\subset C_{l,i}$ and $\mathcal{B}_{l,i}\subset\subset B(p_{l,i},{\rm r_{inj}}-\delta)$ and $d(\mathcal{B}_{l,i}, \partial B(p_{l,i},{\rm r_{inj}}-\delta))>\frac{{\rm r_{inj}}}{4}$. Sets $C_{l,i}$ do not have to be Lipschitz.

Now we improve a family of sets $C_{l,i}$ to a family of sets $C_{l,i}^{'}$ which are constructed as Lipschitz and $\mathcal{B}_{i,l}\subset\subset C_{i,l}^{'}$. That such $C_{l,i}^{'}$ can be chosen in such a way follows from 
$$d(\mathcal{B}_{l,i}^{comp}, \partial B(p_{l,i},{\rm r_{inj}}-\delta))>\frac{{\rm r_{inj}}}{4}.$$
 Again, because there are finitely many sets $C_{l,i}, C_{l,i}^{'}$ all their Lipschitz constants are uniformly bounded. Hence we obtained that all $\mathcal{B}_l\subset\bigcup_{i}C_{l,i}^{'}$. Therefore we covered our "bad" sets $\mathcal{B}_l$ with Lipschitz sets.

Altogether sets $V_l^{''}$ and $C_{l,i}^{'}$ provide the covering of $\Omega$ with Lipschitz sets. Let us rename and renumber the constructed collection of sets to obtain the covering $\{V_l\}$ of $\Omega$ with Lipschitz sets.

Therefore, we can assume that sets $V_{l}$ are Lipschitz and the claim is proven. 

Observe that $\diam(V_{l}^{''}) \approx {\rm r_{inj}}$ for all $l$. It is true, because the image of a ball centered at the origin under $\exp$ is a ball with the same radius and we took $\delta=\frac{1}{10}{\rm r_{inj}}$. Notice that since the radius of any geodesic ball in $\Om$ does not exceed ${\rm r_{inj}}$, each such a ball has volume comparable to $\left(\frac{9}{10}{\rm r_{inj}}\right)^n$. The same happens for sets $C_{i,l}^{'}$. Therefore, the number $m$ of sets $V_{l}$ required to cover $\Omega$ is 
\[
m\approx \left(\frac{{\rm Vol}(\Omega)}{{\rm r_{inj}}^n}\right)^2.  
\]
The implicit constant depends only on set $\Omega$. The dependence lies in the fact that volume of a ball (with a small enough radius) can be bounded by n-th power of radius multiplied by a function depending on dimension and curvature.

Due to our construction on each set $V_l$ we have normal coordinates. This allows us to write the Laplace-Beltrami equation $\Delta_Mu=0$ in local coordinates as follows:
\begin{equation}\label{lapbel}
\frac{1}{\sqrt{\det g}}\sum_{i=1}^{n}\frac{\partial}{\partial x^{i}}\left(\sum_{j=1}^{n}\sqrt{\det g}g^{ij}\frac{\partial u}{\partial x^{j}}\right)=0,
\end{equation}
where $g$ denotes a metric tensor on M.

From this representation we see that equation $\Delta_M u=0$ on every $V_{l}\subset M$ is equivalent to an equation ${\rm div}(A\nabla u)=0$ on $\exp_{p_{l}}^{-1}(V_{l})\subset\mathbb{R}^{n}$, where the square matrix $A\in M^{n\times n}$ depends on a point on $M$ and 
\begin{equation}\label{A-matrix}
A=\sqrt{\det g}(g^{ij}).
\end{equation}
 Let us notice that any set $\exp^{-1}(V_{l})$ is Lipschitz, because the derivative $|D\exp_{p_{l}}|$ is bounded on each $V_{l}$. Moreover, since there are finitely many sets $V_{l}$, the derivatives $D\exp_{p_{l}}$ are uniformly bounded on $\bar{\Omega}$. Therefore, in maps our harmonic equation locally reduces to an $A$-harmonic one on Lipschitz domains in $\mathbb{R}^n$.

At this point we would like to apply~\cite[Theorem 2.15]{bortz-hoffman}, see also~\cite[Theorem 1.3]{hmm}, which provides conditions on an underlying domain  and a bounded  $A$-harmonic function implying its $\varepsilon$-approximability. Namely, the domain has to satisfy the interior corkscrew condition and its boundary is uniformly rectifiable. Here, we study domains
\[
 W_l:=\exp_{p_{l}}^{-1}(V_{l})\subset B(0,R_l-\delta)\subset T_{p_l}M=\mathbb{R}^{n},\quad l=1,\ldots,m.
\]
 Notice that, by Lemma~\ref{lem-Lip-cork}, $V_l$, and hence $W_l$ satisfy the first condition. Furthermore, since all $V_l$, and hence also all $W_l$ by the above discussion, are Lipschitz, the second condition holds because every Lipschitz set is in particular uniformly rectifiable, by direct verification of Definition 2.7 in~\cite{bortz-hoffman} with $\theta $ and $M_0$ depending on Lipschitz constant of $\Omega$ and curvature of $\Omega$.

As for the assumptions on the $A$-harmonic function $u$, matrix $A$ in~\cite[Theorem 2.15]{bortz-hoffman} has bounded coefficients and defines an elliptic operator (cf. Definition 2.4 in~\cite{bortz-hoffman}). This is the case of matrix in~\eqref{A-matrix}, since $W_l$ are bounded, metric $g$ is smooth and positive definite. Moreover, as in~\cite{bortz-hoffman} coefficients of an $A$-harmonic operator are locally Lipschitz, again by smoothness of $g$. What remains to be checked are conditions (2.5) and (2.6) in Definition~\ref{A-harmonic-def}.

In order to check (2.6) we compute the gradient of $A$, cf. ~\eqref{A-matrix}. For $i,j,k=1,\dots,n$ it holds that
\begin{align}
\frac{\partial a_{ij}}{\partial x^{k}}&=\frac{\partial}{\partial x^{k}}\left(\sqrt{\det g}g^{ij}\right) \nonumber\\
&=\frac{1}{2\sqrt{\det g}}\frac{\partial\det g}{\partial x^{k}}g^{ij}+\sqrt{\det g}\frac{\partial}{x^{k}}g^{ij}
\nonumber\\
&=\frac{1}{2\sqrt{\det g}}\det g\,\textnormal{tr}\left(g^{-1}\frac{\partial}{\partial x^{k}}g\right)g^{ij}+\sqrt{\det g}\frac{\partial}{\partial x^{k}}g^{ij} \label{jac-3}\\
&=\sqrt{\det g}\left[\frac{1}{2}\left(\sum_{a,b} g^{ab}\frac{\partial}{\partial x^{k}}g_{ba}\right)g^{ij}+\frac{\partial}{\partial x^{k}}g^{ij}\right], \nonumber
\end{align}
where in \eqref{jac-3} we use the Jacobi formula for derivative of a matrix determinant. By compactness of $\bar{W_l}$ and continuity of $g, g_{ij}, g^{ij}$ and their derivatives we have that 

\begin{equation*}
C_l:=\sup_{x\in \bar{W_l}}\{|g(x)|,|g_{ij}(x)|,|g^{ij}(x)|,|\frac{\partial}{\partial x^k}g_{ij}(x)|,|\frac{\partial}{\partial x^k}g^{ij}(x)|\}<\infty
\end{equation*}

Let further $d_l:=\sup_{x\in\bar{W_l}}\dist(x,\partial W_l)$. Then for all $i,j,k=1,\dots,n$ it holds that
\begin{equation}\label{a_ij}
\left|\frac{\partial a_{ij}}{\partial x^{k}}(x)\right|\le\sqrt{n!C_{l}^{n}}\left(\frac{1}{2}n^2C_{l}^{3}+C_{l}\right)\le\frac{C(n,C_l)}{d_l}\le\frac{\widetilde{C}}{\dist(x,\partial W_{l})}.
\end{equation}
Therefore, (2.6) holds. It remains to prove (2.5). First, since $W_l$ are Lipschitz, then $\partial W_l$ satisfy the Ahlfors--David regularity condition (cf. Definition  2.1 in~\cite{bortz-hoffman}). This observation, together with the above estimates of partial derivatives of $a_{ij}$ imply the following for $x\in\partial W_l$ and $0<r<\diam(\partial W_l)$
\begin{align*}
\frac{1}{H^{n-1}(B(x,r)\!\cap\!\partial W_l)} \int_{B(x,r)\cap W_l}|\nabla A(X)|dX  & \lesssim\frac{1}{r^{n-1}}\int_{B(x,r)\cap W_l}\frac{1}{d_l}dX \\
&\lesssim \frac{1}{r^{n-1}}\int_{B(x,r)\cap W_l}\frac{1}{r}dX \\
&=\frac{1}{r^{n}}\int_{B(x,r)\cap W_l}dX\leq \frac{H^n(B(x,r))}{r^n}\\
&\le C(n). 
\end{align*}
Upon taking the supremum over $x\in\partial W_l$ we arrive at (2.6). In consequence, Theorem 2.15 in~\cite{bortz-hoffman} gives us the $\varepsilon$-approximation of $u$ by BV functions on sets $\exp_{p_{l}}^{-1}(V_{l})$ for $l=1,\ldots, m$. 

In the last step of the proof, we glue together the BV functions constructed above, to obtain one BV function approximating $u$ on $\Omega$. We do it in a following manner. For each $V_{l}$, for $l=1,\ldots, m$, denote by $\phi_{l}$ a BV function approximating $u$ on $V_{l}$. Let $\sigma$ be a permutation of indices $l=1,\dots,m$ satisfying following conditions: The value of permutation $\sigma(1)$, is any number from $1$ to $m$. The value of $\sigma(2)$ is any index $l$ such that $V_{\sigma(1)}\cap V_{\sigma(2)}\neq\emptyset$. Then $\sigma(3)$ is such an index that $V_{\sigma(3)}\cap(V_{\sigma(1)}\cup V_{\sigma(2)})\neq\emptyset$ and so on. Let us denote by $\phi_{\sigma(j)}$ a BV function approximating $u$ on $V_{\sigma(j)}$. On a set $V_{\sigma(2)}\setminus V_{\sigma(1)}$ take a function $\phi_{\sigma(2)}|_{V_{\sigma(2)}\setminus V_{\sigma(1)}}$.  Let function $\phi^1:V_{\sigma(1)}\cup V_{\sigma(2)} \to \mathbb{R}$ be defined as follows
\begin{equation*}
 \phi^1=
	\begin{cases}
\phi_{\sigma(1)}  & \hbox{on }\ V_{\sigma(1)}\\
 \phi_{\sigma(2)} & \hbox{on }\ V_{\sigma(2)}\setminus V_{\sigma(1)}.
	\end{cases}
\end{equation*}
Next choose any set $V_{\sigma(3)}$ that has nonempty intersection with $V_{\sigma(1)}\cup V_{\sigma(2)}$ and a function $\phi_{\sigma(3)}|_{V_{\sigma(3)}\setminus(V_{\sigma(1)}\cup V_{\sigma(2)})}$. Similarly as above, we define function $\phi^2$ as follows:
\begin{equation*}
\phi^2=
\begin{cases}
\phi_{\sigma(1)}  & \hbox{on }\ V_{\sigma(1)}\\
\phi_{\sigma(2)} & \hbox{on }\ V_{\sigma(2)}\setminus V_{\sigma(1)}\\
\phi_{\sigma(3)} & \hbox{on }\ V_{\sigma(3)}\setminus (V_{\sigma(1)}\cup V_{\sigma(2)}).
\end{cases}
\end{equation*}
After finitely many steps we construct a function $\phi=\phi^{m-1}$ defined on $\Omega$. Such a function has bounded variation, since each $\phi_{l}$ has bounded variation and sets $\partial V_{l}\cap\partial V_{k}$ where the function $\phi$ may have additional jumps is of measure zero and there are finitely many of them, so they do not affect the variation of $\phi$. The proof of Lemma~\ref{BV-mfld} is therefore completed.
\end{proof}
It is worth seeing how Laplace-Beltrami equation looks like in the case of an $n$-dimensional sphere.
\begin{example}
Let us consider an n-dimensional sphere $S^n$ and a stereographic projection. Denote the coordinates on a sphere by $(t_1,\dots,t_n)$ and $t^2=\sum t_{i}^{2}$. Then $\Delta_{S^n}u$ reads
\begin{align*}
\Delta_{S^n} u(t_1,\dots,t_n)&=\bigg(\frac{t^2+1}{2}\bigg)^n\sum_{i=1}^{n}\frac{\partial}{\partial t_i}\bigg(\bigg(\frac{2}{t^2+1}\bigg)^{n-2}\frac{\partial u}{\partial t_i}\bigg)\\
&=
\begin{cases}
-\frac{(t^2+1)(n-2)}{2}\sum_{i=1}^{n}t_i\frac{\partial u}{\partial t_i}+\frac{(t^2+1)^2}{4}\sum_{i=1}^{n}\frac{\partial^2 u}{\partial t_i^2} \hspace{5mm} \textnormal{for } n\ge 3,\\
\frac{(t^2+1)^2}{4}\sum_{i=1}^n\frac{\partial^2 u}{\partial t_i^2} \hspace{5mm} \textnormal{for } n=2.
\end{cases}
\end{align*}
Hence, if $u$ is harmonic i.e. $\Delta_{S^n}u=0$, then the associated matrix $A={\rm diag}\left(\left(\frac{2}{t^2+1}\right)^{n-2}\right)$. Set $d_{\Omega}:=\sup_{x\in\Omega}\dist(x,\partial\Omega)\le\diam(\partial\Omega)$ and since $\Omega$ is bounded we have $|t_i|\le K_{\Omega}$ for some constant $K_{\Omega}>0$. By direct computations:
\begin{align*}
\frac{\partial a_{ii}}{\partial t_k}&=(n-2)\left(\frac{2}{1+t^2}\right)^{n-1}t_k,\\
|\nabla A|&=(n-2)\sqrt{n}\left(\frac{2}{t^2+1}\right)^{n-1}t\le (n-2)\sqrt{n}2^{n-1}\sqrt{n}K_{\Omega}\\
&=\frac{(n-2)\sqrt{n}2^{n-1}\sqrt{n}K_{\Omega}d_{\Omega}}{d_{\Omega}}=\frac{C_{\Omega}}{d_{\Omega}}\le \frac{C_{\Omega}}{\dist(x,\partial\Omega)}.
\end{align*}
Therefore, condition (2.5) in Definition \ref{A-harmonic-def} reads:
\begin{align*}
\frac{1}{r^{n-1}}\int_{B(x,r)\cap\Omega}|\nabla A|&\le\frac{1}{r^{n-1}}\int_{B(x,r)\cap\Omega}\frac{C_{\Omega}}{d_{\Omega}}\lesssim\frac{1}{r^{n-1}}\int_{B(x,r)\cap\Omega}\frac{1}{r}\approx\frac{1}{r^n}r^n=1.
\end{align*}
 
\end{example}

\begin{remark}
Notice that everything that was proved so far in this Section applies as much to $A$-harmonic functions. Indeed, in local coordinates equation (\ref{A-harmonic}) takes the form

\begin{equation*}
\frac{1}{\sqrt{\textnormal{det}g}}\sum_i\frac{\partial}{\partial x^i}\left(\sum_l\sum_t\sqrt{\textnormal{det}g}a_{il}(x)g^{lt}\frac{\partial u}{\partial x^t}\right)=0.
\end{equation*}

Equivalently, it can be written as 

\begin{equation}
\frac{1}{\sqrt{\textnormal{det}g}}\textnormal{div}(B(x)\nabla u)=0,
\end{equation}

where $B(x)$ is the matrix with coefficients $b_{ij}(x)=\sum_{l=1}^{n}\sqrt{\textnormal{det}g(x)}(a_{il}(x)g^{lj}(x))$.

Hence, Theorem \ref{thm-e-approx} holds for A-harmonic functions on Lipschitz domains on manifolds. In order to show this, it is enough to prove that $B(x)$ satisfies (2.5) and (2.6).
Let us start with (2.6). We have

\begin{align}
\frac{\partial b_{ij}}{\partial x^{k}}&=\frac{\partial}{\partial x^{k}}\left(\sum_l\sqrt{\det g}a_{il}(x)g^{lj}\right) \nonumber\\
&=\sum_l\left(\left(\frac{\partial}{\partial x^k}\sqrt{\textnormal{det}g}\right)a_{il}(x)g^{lj}+\sqrt{\textnormal{det}g}\frac{\partial a_{il}(x)}{\partial x^k}g^{lj}+\sqrt{\textnormal{det}g}a_{il}(x)\frac{\partial g^{lj}}{\partial x^k}\right)\nonumber\\
&=\sum_l\left(\frac{1}{2\sqrt{\det g}}\det g \textnormal{ tr}\left(g^{-1}\frac{\partial g}{\partial x^k}\right)a_{il}(x)g^{lj}+\sqrt{\det g}\frac{\partial a_{il}(x)}{\partial x^{k}}g^{lj}+\sqrt{\det g}a_{il}(x)\frac{\partial g^{lj}}{x^{k}}\right)\nonumber\\
&=\sqrt{\det g}\sum_l\left[\frac{1}{2}\left(\sum_{a,b}g^{ab}\frac{\partial g_{ba}}{\partial x^k}\right)a_{il}(x)g^{lj}+\frac{\partial a_{il}(x)}{\partial x^{k}}g^{lj}+a_{il}(x)\frac{\partial g^{lj}}{x^{k}}\right].\nonumber
\end{align}

Denote by $M:=\|A\|_{L^{\infty}}(\Omega)$ and let constants $C_l$ be as in the proof of Theorem \ref{thm-e-approx}. Then we have

\begin{align}\label{A-nier}
\left|\frac{\partial b_{ij}}{\partial x^{k}}\right|&\le\sqrt{n!C_{l}^{n}}\sum_l\left|\frac{1}{2}n^2 C_l^3 a_{il}(x)+\frac{\partial a_{il}(x)}{\partial x^{k}}C_l+a_{il}(x)C_l\right|\nonumber\\
&\le\sqrt{n!C_{l}^{n}}n\left(\frac{1}{2}n^2 C_l^3 M+|\nabla A|C_l+MC_l\right)\nonumber\\
&\le\sqrt{n!C_{l}^{n}}n\left(\frac{1}{2}n^2 C_l^3 M+MC_l\right)+\sqrt{n!C_{l}^{n}}nC_l|\nabla A|\nonumber\\
&\lesssim\frac{1}{d}+\frac{1}{\textnormal{dist}(x,\partial W_l)}\\
&\lesssim\frac{1}{\dist(x,\partial W_{l})},\nonumber
\end{align}

where in the inequality (\ref{A-nier}) we use estimate \ref{a_ij}.
Let us now proceed to proving (2.5). We have

\begin{align}
\frac{1}{H^{n-1}(B(x,r)\!\cap\!\partial W_l)}& \int_{B(x,r)\cap W_l}|\nabla B(X)|dX \nonumber\\
&\lesssim\frac{1}{r^{n-1}}\int_{B(x,r)\cap W_l}\frac{1}{\diam W_l}+\sqrt{n!C_l^n}nC_l|\nabla A|dX \nonumber\\
&\lesssim C(n)+C(n,M),
\end{align}

where in the last inequality we use the property (2.5) for $A$ and \ref{a_ij}. 

Therefore, Theorem \ref{thm-e-approx} extends to the setting of A-harmonic functions with $A=A(x)$ as in Definition \ref{A-harmonic-def}.
\end{remark}

\section{Quantitative Fatou property}
The goal of this section is to apply the $\varepsilon$-approximability to prove the Quantitative Fatou Property on Lipschitz domains in Riemannian manifolds.

Since the choice of good maps and their propertie will be important in what follows, let us briefly recall some necessary facts.

If $M$ is a Riemannian manifold, then we always have a chart preserving the Lipschitzness of a set, namely $\exp_p^{-1}$ taken on such set $U\subset M$ that $\exp_p^{-1}$ is a diffeomorphism,see the proof of Theorem \ref{thm-e-approx} above.

However, we can take any chart that preserves bounded Lipschitz sets. Indeed, the following lemma shows that any chart would do as long as Lipschitz domains are 1-connected at the boundary and the image is bounded. Therefore, we do not need to necessarily use exponential maps. Nevertheless, we use them because they are convenient and handy to work with, but any chart with similar properties would be sufficient. By similar properties we mean that:
\begin{itemize}
\item we can take such sets $U_i$ as charts that in each set there is contained a ball with radius uniformly bounded from below,
\item the Lipschitz constants of maps are uniformly bounded from above.
\end{itemize}
The only difference between the above choice of Lipschitz maps and the exponential map is that now different charts will have different Lipschitz constants. However, it only affects the constants in the estimates, which yields that all results are still true.

In the next definition we recall topological notion that plays a crucial role in the studies of the extension of mappings, including the continuous and homeomorphic extensions, see e.g. theorem in \cite[Chapter 2, Section 17]{vaisala}. Moreover, see \cite{adamowicz}, where the notion of prime ends is used to determine the existence of extension.

\begin{definition}
Let $X$ be a metric space. We will say that a set $U\subset X$ is 1-connected at the boundary if for every point $x\in\partial U$ there exists its arbitrarily small neighbourhood $U_x$ such that $U\cap U_x$ is connected.
\end{definition}

An example of a set that is not 1-connected is a slit disc. It is a disc $B(p,r)\subset\mathbb{R}^2$ with a removed line segment joining the center $p$ with a boundary, say at point $x$. Then at point $x$ any neighbourhood $U_x$ has a property that $B(p,r)\cap U_x$ has two connected components for small enough sets $U_x$.

\begin{lemma}\label{good-maps}
Let $U\subset (X,d)$ be an open, connected, precompact and 1-connected at the boundary set. Let further $h:U\rightarrow h(U)$ be a homeomorphism such that $h(U)$ is bounded in $(Y, \tilde{d})$. Then for any bounded Lipschitz subset $U^{\prime}\subset U$ it holds that $h(U^{\prime})$ is also bounded Lipschitz in $h(U)$.
\end{lemma}
The proof of the lemma is in Appendix B.

\begin{remark}
One can approach defining the counting function either independently of maps or in maps. In what follows we take the first approach as it is more natural in the manifold setting. Nevertheless we would like to briefly comment on the approach via maps. Namely, at every boundary point we can choose the local coordinates and in those coordinates define locally the $N$ function as in Definition 3.11. Then we cover $\partial \Om$ with balls of radius $R=\frac{1}{20}{\rm r_{inj}}$ centered at some points $p_i\in\partial\Omega$. Define counting function in set $B(p_1,R)\cap\partial\Omega$ using a chart that preserves Lipschitzness on that set, e.g. $\exp^{-1}_{p_1}$. Then proceed inductively. Define a counting function on set $(B(p_2,R)\setminus B(p_1,R))\cap\partial\Omega$ using a chart on that set, and continue until all boundary is covered. For different choices of charts we obtain deifferent counting functions, but the Quantitative Fatou Property holds for all of them with different constants.
Now, since we already know that a harmonic function defined on $\Omega$ is $\varepsilon$-approximable for every $\varepsilon$ (Section 4) we may apply Lemma 2.9 in \cite{kkpt} and get the Quantitative Fatou Property for Lipschitz domains in complete Riemannian manifolds on balls with radii $r<c<r_{inj}$. 
\end{remark}

We would like to define a counting function without using any chart.
For the readers convenience we recall Definition \ref{counting-function-definition}:

$\Gamma^r(p)$ a cone at some point $p\in\partial\Omega$. Let $u$ be a harmonic function defined on $\Omega$. Denote by $d$ a Riemannian distance in manifold $M$. Fix $\varepsilon>0$, $0<\theta<1$ and $0<r<1$. We will say that a sequence of points $Q_n\in\Gamma^r(p)$ is $(r,\varepsilon,\theta,p)$-admissible for u if 
\begin{align}\label{def-count-function}
|u(Q_n)-u(Q_{n-1})|\ge\varepsilon,\\
d(Q_n,p)<\theta d(Q_{n-1},p).\nonumber
\end{align}
Set 
$$N(r, \varepsilon, \theta)(p)=\sup\{k: \textnormal{there exists an } (r, \varepsilon, \theta, p)\textnormal{-admissible sequence of length }k\}.$$
We will call $N$ a counting function.

We see that a counting function $N$ defined in such a way does not depend on a chosen chart. We would like to prove the Quantitative Fatou Property for such counting function, since it  would be desirable that the QFP is independent of charts and requires only the Riemannian structure.
Recall Definition \ref{cone-definition}.

In what follows, we apply this definition to $r_1$ and  $r_2$ the distances of $p$ to consecutive points in an $(r,\varepsilon,\theta,p)$-admissible sequence, cf. Def. \ref{counting-function-definition} or (\ref{def-count-function}) above. Moreover, if point $p$ is fixed , then we skip writing it and denote $\Gamma_{r_1,r_{2}}:=\Gamma_{r_1,r_{2}}(p)$.

\begin{lemma}\label{lem-o-krzywej}
Let $\Omega\subset M$ be Lipschitz domain and $\Gamma_{r_1,r_2}$, the doubly truncated cone with the aperture $\alpha$, be connected. Let further $x_1\in\Gamma_{r_1,r_2}\cap S(p,r_1)$, $x_2\in\Gamma_{r_1,r_2}\cap S(p,r_2)$ be elements of $(r,\varepsilon,\theta,p)$-admissible sequence corresponding to $r_1$ and $r_2$, respectively. Denote by $\widetilde{\Gamma}_{r_1,r_2}(p)$ a doubly truncated cone with the aperture $\widetilde{\alpha}>\alpha$, so that $\Gamma_{r_1,r_2}\subset\widetilde{\Gamma}_{r_1,r_2}$. Then there exists a curve $\gamma:[r_2,r_1]\rightarrow\widetilde{\Gamma}_{r_1,r_2}$ with $\gamma(r_2)=x_2$ and $\gamma(r_1)=x_1$ and such that $\gamma(r)\in\widetilde{\Gamma}_{r_1,r_2}$ for all $r\in [r_2,r_1]$, with the following properties:
\begin{itemize}
\item $l(\gamma)\le Kd(x_1,x_2)$ for some $K>0$, where $l(\gamma)$ denotes length of $\gamma$,
\item $|\frac{\partial\gamma}{\partial r}|\le C$ for some constant $C$,
\item $\gamma$ intersects every sphere $S(p,r)$ for $r_1<r<r_2$ exactly once.
\end{itemize}
\end{lemma}
The set $\Gamma_{r_1,r_2}$ may fail to be Lipschitz, because at the points of intersection of a cone $\Gamma_r$ and a sphere $S(p,r_2)$ the regularity of the boundary of $\Gamma_{r_1,r_2}$ may worsen, for instance be only H{\"o}lder as some cusps may occur. That is why we take a bigger set $\widetilde{\Gamma}_{r_1,r_2}$.

\begin{proof}
First, by using the exponential map we can reduce the discussion to the ambient space $\mathbb{R}^n$. Notice also that the Lipschitzness of $\Omega$ implies that by taking $\Gamma_R$ with $R$ small enough, we may ensure that $\Gamma_{r_1,r_2}$ is connected for every pair $r_1, r_2$. It is enough to consider such $R$ that for every point $x\in B(p,R)$ the distance $d(x,\partial\Omega)$ is achieved at some point $y\in\partial\Omega\cap B(p,R)$. Such $R$ exists because $\partial\Omega$ is compact and it is Lipschitz. Moreover, such $R$ depends on Lipschitz constant of $\partial\Omega$.

Note that for small enough $R$, the boundary $\partial\Gamma_R$ does not "turn". By turn we mean the following property. Take a tangent space to $\partial\Omega$ at $p$, denoted by $T_p(\partial\Omega)$. Such a space exists at almost every point $p\in\partial\Omega$, because $\partial\Omega$ is a Lipschitz set. Moreover, that space is an $(n-1)$-dimensional subspace of $n$-dimensional tangent space $T_p(M)$. Then any line perpendicular to $T_p(\partial\Omega)$ intersects $\partial\Gamma_{R}$ as long as that line is close enough to point $p$. Furthermore, such a line intersects $\partial\Gamma_{R}$ at least twice: once when it intersects the surface, where $d(q,p)=(1+\alpha)d(q,\partial\Omega)$ and the second time when it intersects sphere $S(p,R)$. However, it can occur that the surface defined by the equation $d(q,p)=(1+\alpha)d(q,\partial\Omega)$ is intersected more than once. If it happens, then we say that that $\partial\Gamma_R$ turns, whereas if there are only two points of intersection we will say that $\partial\Gamma_R$ does not turn. Again, due to compactness and Lipschitz property of $\partial\Omega$ we can ensure that for $R$ small enough $\partial\Gamma_R$ does not turn. 
Moreover, compactness of $\overline{\Omega}$ allows us to choose $R$ small enough satysfing all the aforementioned properties at every point of boundary $\partial\Omega$ i.e. sets $\Gamma_{r_1,r_2}(p)$ are connected for all $p\in\partial\Omega$ and all $0<r_1<r_2\le R$ and $\Gamma_R(p)$ does not turn. 

Next, let us prove the following observation.
\begin{claim}
There exists a constant $\delta_{\alpha}>0$ such that for every point $x\in\partial\Gamma_{r_1,r_2}$ a ball $B(x,\delta_{\alpha} d(x,p))\subset\widetilde{\Gamma}_{(1+\delta_{\alpha})r_1,(1-\delta_{\alpha})r_2}$.
\end{claim}

\begin{proof}
Let $\widetilde{x}\in B(x,\delta_{\alpha} d(x,p))$. Denote by $q$ and $\widetilde{q}$ points on $\partial\Omega$ such that $d(x,\partial\Omega)$ and $d(\widetilde{x},\partial\Omega)$ are attained, i.e. $d(q,x)=d(x,\partial\Omega)$ and $d(\widetilde{q},\widetilde{x})=d(\widetilde{q},\partial\Omega)$, respectively. Then
\begin{align*}
d(x,\partial\Omega)&\le d(x,\widetilde{q})\le d(\widetilde{x},\partial\Omega)+d(x,\widetilde{x})\\
&\le d(\widetilde{x},\partial\Omega)+\delta_{\alpha} d(x,p)\\
&=d(\widetilde{x},\partial\Omega)+\delta_{\alpha}(1+\alpha)d(x,\partial\Omega),
\end{align*}

where the equality is the consequence of $x\in\partial\Gamma_{r_1,r_2}$ and so, in particular $x$ satisfies the equation of the boundary of $\Gamma(p)$.

Hence
\begin{align*}
(1-\delta_{\alpha}(1+\alpha))d(x,\partial\Omega)\le d(\widetilde{x},\partial\Omega).
\end{align*}

Note, that for this inequality to make sense, $\delta_{\alpha}<\frac{1}{1+\alpha}$. Now we can estimate the distance of $\widetilde{x}$ to vertex $p$:
\begin{align*}
d(\widetilde{x},p)&\le d(x,p)+d(\widetilde{x},x)\le (1+\delta_{\alpha})d(x,p)\\
&=(1+\delta_{\alpha})(1+\alpha)d(x,\partial\Omega)\\
&\le\frac{(1+\delta_{\alpha})(1+\alpha)}{1-\delta_{\alpha}(1+\alpha)}d(\widetilde{x},\partial\Omega)\\
&=\left( 1+\frac{\alpha+2\delta_{\alpha}(1+\alpha)}{1-\delta_{\alpha}(1+\alpha)}\right) d(\widetilde{x},\partial\Omega).
\end{align*}
 Therefore, in order to make sure that ball $B(x,\delta_{\alpha} r)\subset\widetilde{\Gamma}_{r_1,r_2}$ we need to find $\delta_{\alpha}$ such that
\begin{align*}
\frac{\alpha+2\delta_{\alpha}(1+\alpha)}{1-\delta_{\alpha}(1+\alpha)}\le\widetilde{\alpha}
\end{align*}

which gives
\begin{align*}
\delta_{\alpha}\le\frac{\widetilde{\alpha}-\alpha}{(1+\alpha)(2+\widetilde{\alpha})}<\frac{1}{1+\alpha}
\end{align*}
and completes the proof of the claim.
\end{proof}

We now show that there exist certain two-dimensional quasirectangles contained in $\widetilde{\Gamma}_{r_1,r_2}$ which enable to choose curve $\gamma$ in such a way that $|\frac{\partial\gamma}{\partial r}|$ is uniformly bounded in $\widetilde{\Gamma}_{r_1,r_2}$. Recall that by assumption $x_1$ and $x_2$ are given points such that  $x_1\in\Gamma_{r_1,r_2}\cap S(p,r_1)$, $x_2\in\Gamma_{r_1,r_2}\cap S(p,r_2)$. Let $l_1$ and $l_2$ denote line segments beginning at $p$ and crossing $x_1$ and $x_2$ respectively. Let $L_{12}$ denote a two-dimensional cone spanned between $l_1$ and $l_2$. Let $\hat{r}=(1-\delta_{\alpha})r_1$. We would like to show now that a quasirectangle $K_{r_1,\hat{r}}:=L_{12}\cap(B(p,r_1)\setminus B(p,\bar{r}))\subset\widetilde{\Gamma}_{r_1,r_2}$ .

First we need to know that $\hat{r}>r_2$. Since we know that $r_2\le \theta r_1$ it is enough to take $\delta_{\alpha}<1-\theta$. By abuse of notation take $\delta_{\alpha,\theta}={\rm min}\{\delta_{\alpha},1-\theta\}$ and change, if necessary, $\hat{r}=(1-\delta_{\alpha,\theta})r_1$. For every point in $K_{r_1,\hat{r}}$ there exists a line segment $l_{\widetilde{x}}$ joining point $p$ with some point $\widetilde{x}\in\partial\Gamma_{\hat{r}}$ such that this point lies on $l_{\widetilde{x}}$. Denote by $\widetilde{r}:=d(\widetilde{x},p)$. Since by the previous step of the proof we know that a ball $B(\widetilde{x},\delta_{\alpha,\theta}\widetilde{r})\subset\widetilde{\Gamma}_{r_1,r_2}$ it suffices to observe that $l_{\widetilde{x}}\subset B(\widetilde{x},\delta_{\alpha,\theta}\widetilde{r})$. Indeed, since $d(p,K_{r_1,\hat{r}})>\hat{r}$, it is therefore enough to show that
\begin{align*}
\widetilde{r}-\hat{r}=\widetilde{r}-(1-\delta_{\alpha,\theta})r_1<\delta_{\alpha,\theta}\widetilde{r}
\end{align*}

which is trivially equivalent to $\widetilde{r}< r_1$ which is always true.

Let us now construct a curve $\gamma$ as in the assertion of the Lemma. It consists of two subcurves. First one, denoted by $\gamma_1$, is contained in a line segment starting at $p$ and containing $x_2$ and the second one, denoted by $\gamma_2$, is contained in a quasirectangle $K_{r_1,(1-\delta_{\alpha,\theta})r_1}$ between $r_1$ and $(1-\delta_{\alpha,\theta})r_1$. Moreover, we can choose $\gamma_2$ in such a way that its derivative is bounded. Indeed, on $\gamma_1$ it holds that $|\frac{\partial\gamma_1}{\partial r}|=1$, while on $\gamma_2$ we can estimate that
\begin{align*}
\left|\frac{\partial\gamma_2}{\partial r}\right|\le 1+\frac{\alpha}{\delta_{\alpha,\theta}}.
\end{align*} 

To see the above estimate take as $\gamma_2$ a quasidiagonal of quasirectangle. By this, we mean a curve that in polar coordinates in the plane $L_{12}$ with point $p$ as $0$ is given by 
\[ \gamma(r)=(r,\phi_2+\frac{\phi_1-\phi_2}{\delta_{\alpha,\theta} r_1}(r-(1-\delta_{\alpha,\theta})r_1))\]
 with $r\in [(1-\delta_{\alpha,\theta})r_1,r_1]$, where $x_1=(r_1,\phi_1)$ and $x_2=(r_2,\phi_2)$. This curve stars at the endpoint of $\gamma_1$ and ends at $x_1$. One gets that $\frac{\partial\gamma_2}{\partial r}=(1,\frac{\phi_1-\phi_2}{\delta_{\alpha,\theta} r_1})$. Hence 
\[\left|\frac{\partial\gamma_2}{\partial r}\right|=\sqrt{1+r^2\left(\frac{\phi_1-\phi_2}{\delta_{\alpha,\theta} r_1}\right)^2}\le\sqrt{1+\frac{\alpha^2}{\delta_{\alpha,\theta}^2}}\le 1+\frac{\alpha}{\delta_{\alpha,\theta}}.\]

Thus, the derivative with respect to $r$ is bounded on both curves. Furthermore, we can choose $\gamma_2$ such that its length with respect to Euclidean distance $l(\gamma_2)<2\pi r_1$, since any two points on different concentric spheres can be connected by a curve of length smaller than perimeter of a bigger of those two spheres and quasidiagonal is such a curve. Quasidiagonal also intersects every sphere centered at $p$ with radius between $r_2$ and $r_1$ exactly once.

Finally we estimate the Euclidean length of $\gamma$
\begin{align}\label{q}
l(\gamma)&=l(\gamma_1)+l(\gamma_2)\le ((1-\delta_{\alpha,\theta})r_1-r_2)+2\pi r_1\nonumber\\
&\le (r_1-r_2)+2\pi r_1\le (r_1-r_2)+\frac{2\pi}{1-\theta}(r_1-r_2)\\
&\le \left(1+\frac{2\pi}{1-\theta}\right)d(x_1,x_2),\nonumber
\end{align}

where in (\ref{q}) we use that $x_1$ and $x_2$ belong to the $(r,\varepsilon,\theta,p)$-admissible sequence and so $r_2<\theta r_1$.
Let us notice that constants in all estimates depend only on $\alpha$, $\widetilde{\alpha}$, $\theta$ and the Lipschitz constant of the exponential map. There is no dependence on $r_1$ and $r_2$.
\end{proof}
\begin{remark}\label{rem-small}
In order to apply Lemma \ref{lem-o-krzywej} we need $r$ to be sufficiently small. Fortunately for the Quantitative Fatou Property it is necessary to know only the behaviour of harmonic function $u$ close to boundary $\partial\Omega$. Therefore, it is not a problem that we need to restrict possible $r$, as long as we can find uniformly some radius $r$ for all boundary points such that every $\Gamma_r$ satisfies all our assumptions at every point $p$ of $\partial\Omega$. Since $\Omega$ is compact, it can be achieved.
\end{remark}
If $r$ is small enough, i.e. $r<r_{\textnormal{inj}}$, then $\Gamma_r(p)$ is contained in a ball centered at $p$ such that there are local coordinates due to the exponential map which is a bounded diffeomorphism on that ball. Therefore, for sufficiently small $r$ we can always assume that the ambient space is Euclidean. 
\begin{lemma}\label{lemat-o-calkowaniu}
Let $\Omega\subset M$ be a Lipschitz domain and $u:\Omega\rightarrow\mathbb{R}$ be a bounded harmonic function with $\|u\|_{\infty}\le 1$. Suppose that $\varepsilon>0$ and $\phi$ is an $\frac{\varepsilon}{4}$-approximation of a bounded harmonic function $u$. If the counting function $N(r,\varepsilon,\theta)(p)\ge k$ for some $k\in\mathbb{N}$, then the following holds
\begin{equation}\label{lemat.5.5}
\int_{\Gamma_r(p)}\frac{|\nabla\phi(x)|}{d(x,p)^{n-1}}dx\ge kC_{\varepsilon,\theta}.
\end{equation}
\end{lemma}
\begin{proof}
Without loss of generality we may assume that $\Omega\subset\mathbb{R}^n$, see Remark \ref{rem-small}. Let us also assume that $p=0$. Since, by assumption $N(r,\varepsilon,\theta)(p)\ge k$ there is a finite sequence of points $x_1,\dots, x_k\in\Gamma_r(0)$ such that
\[
 0<|x_k|<\dots<|x_1|<r, |x_{j+1}|\le\theta|x_j| \hspace{5mm} \textnormal{for } j=1,\dots,k-1 
\]
and
\[
 |u(x_j)-u(x_{j+1})|\ge\varepsilon. 
\]
Since $u$ after composing with exponential map is Lipschitz, so in particular is H{\"o}lder continuous and bounded, there exists $\delta=\delta(\varepsilon)>0$ such that
\[
|u(x)-u(x_j)|<\frac{\varepsilon}{8} \hspace{5mm} \textnormal{for } x\in l_j:=\{y\in\Gamma_r(0)\cap S(0,|x_j|): d_{S_j}(y,x_j)<\delta|x_j|\},
\]
 where $d_{S_j}$ denotes a distance on a sphere $S(0,|x_j|)$. Therefore, for all $x\in l_j$ and $y\in l_{j+1}$ holds that: $|u(x)-u(y)|\ge\frac{3\varepsilon}{4}$. Moreover, we have $|\phi(x)-\phi(y)|\ge\frac{\varepsilon}{4}$.

Let $U_j$ be a doubly truncated Euclidean cone such that its angle is $2\delta$, its vertex is $0$ and for every $z\in U_j$ the following holds: $r_j\le|z|\le r_{j-1}$. Let $\gamma_j\subset U_j$ be a curve given by the assertion of Lemma \ref{lem-o-krzywej}. Consider the transformation $F_j: U_j\rightarrow\mathbb{R}^n$ with the following properties:
\begin{enumerate}
\item the image of a symmetry axis of $U_j$, denoted by $l_{U_j}$ is $\gamma_j$, i.e. $F_j(l_{U_j})=\gamma_j$, $F_j(x)=\gamma_j(|x|)$  for every  $x\in l_{U_j}$.
\item for every $r$ it holds that $F_j|_{U_j\cap S(0,r)}$ is a rotation such that a point on a symmetry axis is transformed into $\gamma_j(r)$.
\end{enumerate}
Such $F_j$ is piecewise smooth, because $\gamma_j$ is piecewise smooth. Furthermore, $F_j$ does not change the volume of a set $U_j$, and hence the absolute value of its Jacobi determinant equals $1$. To see this claim let $U\subset U_j$ be measurable and compute that
$$
{\rm Vol}(F(U)):=\int_{F(U)}1dx=\int_{r_j}^{r_{j-1}}\int_{F_j(U)\cap S(0,r)}1dH^{n-1}dr.
$$
Here we apply the coarea formula  with function $f(x)=|x|$, see \cite[Chapter 3.4]{evans-gariepy}. Moreover, the Jacobian of $f$ equals $1$, see \cite[Chapter 3.2]{evans-gariepy} for the definition of the Jacobian of a real-valued function. Since the $(n-1)$-Hausdorff measure on a sphere is rotation invariant we get
$$
\int_{r_j}^{r_{j-1}}\int_{F_j(U)\cap S(0,r)}1dH^{n-1}dr=\int_{r_j}^{r_{j-1}}\int_{U\cap S(0,r)}1dH^{n-1}dr=\int_{U}1dx,
$$
where the latter equality follows again from the coarea formula. Hence for every measurable set $U\subset U_j$ we have
$$
{\rm Vol}(F(U))={\rm Vol}(U).
$$
Notice that $F_j(U_j\cap S(0,r_{j}))=l_{j}$ and $F_j(U_j\cap S(0,r_{j+1}))=l_{j+1}$. It follows that 
\begin{equation}\label{nier-1-4}
\left|\int_{r_{j}}^{r_{j-1}}\frac{\partial}{\partial r}\phi(F_j)dr\right|\ge\frac{\varepsilon}{4}.
\end{equation}

Since $F_j$ is given by a rotation, its partial derivative with respect to $r$ is solely determined by $\frac{\partial\gamma}{\partial r}$. However, due to Lemma \ref{lem-o-krzywej} we know that $|\frac{\partial\gamma}{\partial r}|\le 1+\frac{\alpha}{\delta_{\alpha,\theta}}$ and hence $|\frac{\partial}{\partial r}F_j|\le 1+\frac{\alpha}{\delta_{\alpha,\theta}}$. 
We are now in a position to show assertion (\ref{lemat.5.5}). It holds that:
\begin{align*}
\int_{\Gamma_{r_j,r_{j+1}}}\frac{|\nabla\phi(x)|}{|x|^{n-1}}dx\ge\int_{F_j(U_j)}\frac{|\nabla\phi(x)|}{|x|^{n-1}}dx=\int_{U_j}\frac{|\nabla\phi (F_j(\widetilde{x}))|}{|\widetilde{x}|^{n-1}}d\widetilde{x},
\end{align*}
as $F(U_j)\subset\Gamma_{r_j,r_{j+1}}$ and by the change of variables formula.

Let us notice that by the chain rule we have
\[
\nabla(\phi\circ F_j)(\widetilde{x})=(\textnormal{ad} DF_j)(\widetilde{x})\nabla\phi(F_j(\widetilde{x})),
\]
where by $(\textnormal{ad}DF_j)(\widetilde{x})$ we mean the adjoint operator of $DF_j(\widetilde{x})$ defined via the scalar product given by the Riemannian metric $g$:
\begin{align*}
g((DF_j)(\widetilde{x}) X,Y)=g(X,(\textnormal{ad}DF_j )(\widetilde{x})Y)
\end{align*}
for all $X,Y\in T_{\widetilde{x}}M^n$. In the spherical coordinates $(r,\phi_1,\dots,\phi_{n-1})$ on $\Omega$ it holds that
\[
\frac{\partial}{\partial r}(\phi\circ F_j)(\widetilde{x})=\langle r_1((\textnormal{ad}DF_j)(\widetilde{x})),\nabla\phi(F_j(\widetilde{x}))\rangle,
\]
where $r_1((\textnormal{ad}DF_j)(\widetilde{x}))$ stands for the first row of matrix $(\textnormal{ad}DF_j)(\widetilde{x})$ and $\langle\cdot , \cdot\rangle$ denotes the Euclidean scalar product. Hence we get
\begin{equation}\label{nier-r-1}
|\nabla\phi(F_j(\widetilde{x}))|\ge\frac{|\frac{\partial}{\partial r}(\phi\circ F_j)(\widetilde{x})|}{|r_1((\textnormal{ad}DF_j)(\widetilde{x}))|}.
\end{equation}
Recall that in spherical coordinates metric $g$ is given by the following matrix
\begin{align*}
g=\begin{bmatrix}
	1 &  &  &  & \\
	   & r^2 &  &  & \\
	   & & r^2 \sin^2(\phi_1) & & \\
	   & & & \ddots & \\
	   & & & & r^2 \sin^2(\phi_1)\cdot\ldots\cdot\sin^2(\phi_{n-2})	
    \end{bmatrix},
\end{align*}
where $r$ stands for $|\widetilde{x}|$.
Therefore after calculation one obtains that 
\[
r_1((\textnormal{ad}DF_j)(\widetilde{x}))=\left(\frac{\partial}{\partial r}F^1_j,r^2\frac{\partial}{\partial r}F^2_j,r^2\sin^2\phi_1\frac{\partial}{\partial r}F^3_j,\dots,r^2\sin^2\phi_1\cdot\ldots\cdot\sin^2\phi_{n-2}\frac{\partial}{\partial r}F^n_j\right).
\]
Note that 
\[
|r_1((\textnormal{ad}DF_j)(\widetilde{x}))|\le\left|\frac{\partial}{\partial r}F_j(\widetilde{x})\right|\sqrt{1+(n-1)^2r^4}\lesssim(n-1)\sqrt{2}\left|\frac{\partial}{\partial r}F_j(\widetilde{x})\right|,
\]
where $|\cdot|$ stands for the length of a vector with respect to scalar product $g$.
Therefore, by (\ref{nier-r-1}) we have
\begin{align}
\int_{U_j}\frac{|\nabla\phi (F_j(\widetilde{x}))|}{|\widetilde{x}|^{n-1}}d\widetilde{x}&\ge \int_{U_j}\frac{1}{|r_1((\textnormal{ad}DF_j)(\widetilde{x}))|}\frac{|\frac{\partial}{\partial r}(\phi\circ F_j)(\widetilde{x})|}{|\widetilde{x}|^{n-1}}d\widetilde{x}\nonumber\\
&\gtrsim\frac{1}{(n-1)\sqrt{2}}\int_{U_j}\frac{1}{|\frac{\partial}{\partial r}DF_j)(\widetilde{x}))|}\frac{|\frac{\partial}{\partial r}(\phi\circ F_j)(\widetilde{x})|}{|\widetilde{x}|^{n-1}}d\widetilde{x}.\label{r1}
\end{align}
Hence, due to Lemma \ref{lem-o-krzywej} we obtain 
\begin{equation}\label{r2}
\int_{U_j}\frac{1}{|\frac{\partial}{\partial r}F_j(\widetilde{x})|}\frac{|\frac{\partial}{\partial r}(\phi\circ F_j)(\widetilde{x})|}{|\widetilde{x}|^{n-1}}d\widetilde{x}\gtrsim \frac{\delta_{\alpha,\theta}}{\delta_{\alpha,\theta}+\alpha}\int_{U_j}\frac{|\frac{\partial}{\partial r}(\phi\circ F_j)(\widetilde{x})|}{|\widetilde{x}|^{n-1}}d\widetilde{x}.
\end{equation}
Since $U_j$ is measurable and the function that we want to integrate is integrable we are allowed to use coarea formula with Lipschitz function $f:\mathbb{R}^n\rightarrow\mathbb{R}$ given by $f(\widetilde{x})=|\widetilde{x}|=t$. Therefore,
\begin{align*}
\int_{U_j}\frac{|\frac{\partial}{\partial r}(\phi\circ F_j)(\widetilde{x})|}{|\widetilde{x}|^{n-1}}d\widetilde{x}=\int_{r_{j}}^{r_{j-1}}\int_{U_j\cap S(0,r)}\frac{|\frac{\partial}{\partial r}(\phi\circ F_j)(\omega_r)|}{r^{n-1}}dH^{n-1}(\omega_r)dr,
\end{align*}
where $\omega_r$ denote points on set $U_j\cap S(0,r)$.
By the change of variables $\omega_r\mapsto\frac{\omega_r}{r}=\omega$, we scale every sphere to a sphere $S(0,1)$ of radius $1$ and obtain
\begin{align*}
\int_{r_{j}}^{r_{j-1}}\int_{U_j\cap S(0,r)}\frac{|\frac{\partial}{\partial r}(\phi\circ F_j)(\omega_r)|}{r^{n-1}}dH^{n-1}(\omega_r)dr=\int_{r_{j}}^{r_{j-1}}\int_{A}\frac{|\frac{\partial}{\partial r}(\phi\circ F_j)(r\omega)|}{r^{n-1}}r^{n-1}dH^{n-1}(\omega)dr,
\end{align*}
where $A=\{x\in S(0,1): d_{S(0,1)}(y,x)<\delta\}$ for some $y\in S(0,1)$. The $H^{n-1}$-measure of $A$ is independent of choice of $y$. Recall that $U_j$ is a doubly truncated cone. Hence set $A$ is just a radial projection of that cone on a sphere with radius $1$ and point $y$ only denotes the projection of its axis. Now we can use the Fubini theorem to change the order of integration to get
\begin{align*}
\int_{r_{j}}^{r_{j-1}}\int_{A}\frac{|\frac{\partial}{\partial r}(\phi\circ F_j)(r\omega)|}{r^{n-1}}r^{n-1}dH^{n-1}(\omega)dr=\int_A\int_{r_{j}}^{r_{j-1}}\left|\frac{\partial}{\partial r}(\phi\circ F_j)(r\omega)\right|drdH^{n-1}(\omega).
\end{align*} 
This together with (\ref{r1}), (\ref{r2}) and (\ref{nier-1-4}) imply the following
\begin{align*}
\int_{U_j}\frac{|\nabla\phi(F_j(\widetilde{x}))|}{|\widetilde{x}|^{n-1}}d\widetilde{x}&\ge\frac{\delta_{\alpha,\theta}}{\delta_{\alpha,\theta}+\alpha}\int_A\int_{r_{j}}^{r_{j-1}}\left|\frac{\partial}{\partial r}(\phi\circ F_j)(t\omega)\right|dtdH^{n-1}(\omega)\\
&\ge \frac{\delta_{\alpha,\theta}}{\delta_{\alpha,\theta}+\alpha}\int_A\frac{\varepsilon}{4}dH^{n-1}\approx C(n,\varepsilon,\alpha,\theta).
\end{align*} 
Recall, by the discussion at the beginning of the proof that $\delta=\delta(\varepsilon)$. Now it is enough to sum over $j=1,\dots,k-1$ to get the assertion of a Lemma.
\end{proof}

Recall that ${\rm r_{inj}}(\Omega)$ denotes the infimum of injectivity radii taken over set $\Omega$. When $\Omega$ is fixed, we will write ${\rm r_{inj}}:={\rm r_{inj}}(\Omega)$ for the sake of simplicity of the notation.
We are now ready to prove one of the key results of our work, namely the Quantitative Fatou Theorem for harmonic functions on Riemannian manifolds.

\begin{thm}\label{glowne twierdzenie}
Let $M$ be a complete Riemannian manifold and let further $\Omega\subset M^n$ be a Lipschitz domain. Furthermore, let $u:\Omega\rightarrow\mathbb{R}$ be a harmonic bounded function with $\|u\|_{\infty}\le 1$. Then for every point $p\in\partial\Omega$
\begin{align*}
\sup_{\substack{0<r<r_{inj}}}\frac{1}{r^{n-1}}\int_{\partial\Omega\cap B(p,r)}N(r,\varepsilon,\theta)(q)d\sigma(q)\le C(\varepsilon,\alpha,\theta,n,\Omega),
\end{align*}

where $\varepsilon,\alpha,\theta$ are constants in the definition of the counting function. In particular, constant $C$ is a independent of $u$. 
\end{thm}

\begin{proof}

Define a shadow of a point $x\in\Omega$ as follows 
\[
S(x):=\{q\in\partial\Omega: x\in\widetilde{\Gamma}(q)\}=\partial\Omega\cap B(x,(1+\alpha)d(x,\partial\Omega)).
\]
The definition of the shadow and the cone are related as follows:
\begin{align}\label{eq1}
x\in\widetilde{\Gamma}(q)\Leftrightarrow q\in S(x).
\end{align}
Let us first estimate the following integral
\begin{align*}
\int_{\partial\Omega\cap B(p,r)}\int_{\widetilde{\Gamma}^r(q)}&|\nabla\phi(x)|d(x,q)^{1-n}dxd\sigma(q)\\
&=\int_{\partial\Omega\cap B(p,r)}\int_{\Omega\cap B(p,2r)}|\nabla\phi(x)|d(x,q)^{1-n}\chi_{\widetilde{\Gamma}^r(q)}(x)dxd\sigma(q).
\end{align*}
In this equality we replace integration over $\widetilde{\Gamma}^r(q)$ with integration over $\Omega\cap B(p,2r)$ with characteristic function of $\widetilde{\Gamma}^r(q)$, as every truncated cone $\widetilde{\Gamma}_r$ is contained in a ball with radius $2r$. 

Now we use the Fubini theorem to change the order of integration:
\begin{align*}
\int_{\partial\Omega\cap B(p,r)}\int_{\Omega\cap B(p,2r)}&|\nabla\phi(x)|d(x,q)^{1-n}\chi_{\widetilde{\Gamma}^r(q)}(x)dxd\sigma(q)\\
&=\int_{\Omega\cap B(p,2r)}|\nabla\phi(x)|\int_{\partial\Omega\cap B(p,r)}d(x,q)^{1-n}\chi_{\widetilde{\Gamma}^r(q)}(x)d\sigma(q)dx\\
&\le\int_{\Omega\cap B(p,2r)}|\nabla\phi(x)|\int_{\partial\Omega\cap B(p,r)}d(x,q)^{1-n}\chi_{S(x)}(q)d\sigma(q)dx,
\end{align*}
where in the last inequality we used (\ref{eq1}) and the fact that $\widetilde{\Gamma}_r(q)\subset\widetilde{\Gamma}(q)$.
Furthermore, since $d(x,q)\ge d(x,\partial\Omega)$ we have the following estimate:
\begin{align*}
\int_{\Omega\cap B(p,2r)}|\nabla\phi(x)|\int_{\partial\Omega\cap B(p,r)}&d(x,q)^{1-n}\chi_{S(x)}(q)d\sigma(q)dx\\
&\le \int_{\Omega\cap B(p,2r)}|\nabla\phi(x)|d(x,\partial\Omega)^{1-n}\int_{\partial\Omega\cap B(p,r)}\chi_{S(x)}(q)d\sigma(q)dx.
\end{align*}
Next we need the following observation.
For every $x\in\Omega$, it holds that $S(x)\subset\partial\Omega\cap B(q_x,(2+\widetilde{\alpha})d(x,\partial\Omega))$, where $q_x$ denotes a point in $\Omega$ where $d(x,\partial\Omega)$ is attained. Indeed, let $y \in S(x)$. Then
\begin{align*}
d(y,q_x)&\le d(x,q_x)+d(x,y)=d(x,\partial\Omega)+d(x,y)\le (2+\widetilde{\alpha})d(x,\partial\Omega).
\end{align*}
Therefore, 
\begin{align}\label{r3}
\nonumber \int_{\Omega\cap B(p,2r)}&|\nabla\phi(x)|d(x,\partial\Omega)^{1-n}\int_{\partial\Omega\cap B(p,r)}\chi_{S(x)}(q)d\sigma(q)dx\\
&\nonumber \le\int_{\Omega\cap B(p,2r)}|\nabla\phi(x)|d(x,\partial\Omega)^{1-n}\int_{\partial\Omega\cap B(p,r)}\chi_{\partial\Omega\cap B(q_x,(2+\widetilde{\alpha})d(x,\partial\Omega))}(q)d\sigma(q)dx\\
&\nonumber \lesssim_{\widetilde{\alpha}} \int_{\Omega\cap B(p,2r)}|\nabla\phi(x)|d(x,\partial\Omega)^{1-n}d(x,\partial\Omega)^{n-1}dx\\
&=\int_{\Omega\cap B(p,2r)}|\nabla\phi(x)|dx\le C(\Omega) (2r)^{n-1}.
\end{align}
where in the second inequality we apply the Ahlfors-David regularity of $\partial\Omega$ which gives that
\[
\sigma(\partial\Omega\cap B(q_x,(2+\widetilde{\alpha})d(x,\partial\Omega)))\lesssim (2+\widetilde{\alpha})^{n-1}(d(x,\partial\Omega))^{n-1},
\]
while in the last inequality we use the fact that $\phi$ is $\varepsilon$-approximation of $u$.
Finally, by Lemma \ref{lemat-o-calkowaniu} and (\ref{r3}) we get the assertion of the theorem
\begin{align*}
C_{\varepsilon,\theta}\int_{\partial\Omega\cap B(p,r)} N(r,\varepsilon,\theta)(q)d\sigma(q)&\le\int_{\partial\Omega\cap B(p,r)}\int_{\widetilde{\Gamma}^r(q)}|\nabla\phi(x)|d(x,q)^{1-n}dxd\sigma(q)\\
&\le C(\Omega) (2r)^{n-1}\lesssim C(\varepsilon,\alpha,\theta,n,\Omega)r^{n-1},
\end{align*}
which proves the theorem.

\end{proof}

\appendix 
\section{Lipschitz sets satisfy the interior corkscrew condition}

\begin{proof}[Proof of Lemma~\ref{lem-Lip-cork}]
Let $Z$ be any bounded Lipschitz set in $\mathbb{R}^{n}$. By definition of the Lipschitz set, for each $z\in\partial Z$ there are a hyperplane $H$ such that $z\in H$ and numbers $\tilde{r}$, $h$  with a cylinder $C=\{x+y{\bf n}:x\in B(z,\tilde{r})\cap H, -h<y<h\}$ and a Lipschitz function $g:H\rightarrow\mathbb{R}$ such that 
\begin{enumerate}
\item $Z\cap C=\{ x+y\,{\bf n}: x\in B(z,\tilde{r})\cap H, -h<y<g(x)\},$
\item $\partial Z\cap C=\{x+y{\bf n}:x\in B(z,\tilde{r})\cap H, y=g(x)\},$
\end{enumerate}
where ${\bf n}$ is a unit vector normal to $H$ that is outer with respect to $Z$. If at point $z\in\partial Z$ the boundary is of class $C^1$ we take as a hyperplane $H$ a tangent one at $z$. In another case take any hyperplane that satisfies aforementioned conditions.
In other words, there is a cone contained in $Z$ with vertex $z$, an angle $\alpha$ such that $\tan{\alpha}=-\frac{2L}{1-L^2}$, where $L$ is a Lipschitz constant of $g$, and height $h$. Since $\partial Z$ is compact there exist minimal $h$, denoted by $\tilde{H}$, minimal $\tilde{r}$, denoted by $\tilde{R}$, and maximal Lipschitz constant, denoted by $\tilde{L}$, such that any cone with vertex in $\partial Z$ and parameters given by $\tilde{H}, \tilde{R}$ and $\tilde{L}$ is contained in $Z$. Let us denote a cone with such parameters and vertex at $z$ by $K(z)$. We want to show interior corkscrew condition, i.e. we want to show that there is a constant $c>0$ such that for each $z\in\partial Z$ and each $0<r<\diam(Z)$ there exists a point $\tilde{z}\in Z\cap B(z,r)$ such that $B(\tilde{z},cr)\subset Z\cap B(z,r)$. 
For a point $z\in\partial Z$ put $\tilde{z}=z-\frac12\min(\tilde{H},r){\bf n}$. We notice that the distance from $\tilde{z}$ to lateral surface of a cone $K(z)$ is given by $\frac{\min(\tilde{H},r)}{2\sqrt{1+\tilde{L}^2}}$ and the distance of $\tilde{z}$ from a base of a cone $K(z)$ is given by $\tilde{H}-\frac12\min(\tilde{H},r)$. We want to find a constant $F$ such that the ball $B(\tilde{z},Fr)$ is contained in both ball $B(z,r)$ and cone $K(z)$. To ensure that a ball with radius $Fr$ is contained in a cone $K(z)$ the following inequalities have to be satisfied:
\begin{align*}
Fr<d=\frac{\min(r,\tilde{H})}{2\sqrt{1+\tilde{L}^2}},\\
Fr<\tilde{H}-\frac12\min(r,\tilde{H}).
\end{align*}
The condition needed for ball with radius $Fr$ to be contained in $B(z,r)$ is:
\begin{align*}
Fr+\frac12\min(r,\tilde{H})<r
\end{align*}
Let us put $F=\frac12\frac{1}{2\sqrt{1+\tilde{L}^2}}\frac{\tilde{H}}{\diam Z}$. One can check that such $F$ satisfies all necessary conditions. Thus, it holds that a ball $B(\tilde{z}, Fr)$ is contained both in a cone $K(z)$ and in a ball $B(z,r)$. Hence $B(\tilde{z},Fr)\subset Z\cap B(z,r)$. 
\end{proof}

\section{Proof of Lemma \ref{good-maps}}

First we prove that $h$ can be extended to $\overline{U}$. Take $x\in\partial U$ and any sequence $x_n\in U$ that converges to $x$. We would like to define $h(x)$ as $\lim h(x_n)$. We have to check whether $h(x_n)$ converges. Because $h(U)$ is bounded, we can take a convergent subsequence $h(x_{n_k})$ and denote its limit as $y$. Notice that $y\in\overline{h(U)}$. Suppose that there is another convergent subsequence $h(x_{n_l})$ and it has a different limit $\tilde{y}$. Since $y$ and $\tilde{y}$ are distinct we can find their disjoint neighbourhoods $V$ and $\tilde{V}$ such that almost all of $h(x_{n_k})$ and $h(x_{n_l})$ are in $V$ and $\tilde{V}$ respectively. Now we take intersections of $V$ and $\tilde{V}$ with $h(U)$ and notice that their counterimages under $h$ are also disjoint subsets in $U$. However, in these counterimages we heve points $x_{n_k}$ and $x_{n_l}$ respectively. There are two possibilities. Either we found subsequences of $x_n$ which converge to different limits, but it is a contradiction with convergence of $x_n$, or both $x_{n_k}$ and $x_{n_l}$ converge to $x$ and $h^{-1}(V)$ and $h^{-1}(\tilde{V})$ are disjoint and point $x$ belongs to both of their boundaries. But since there is a connected neighbourhood of $x$ in $U$, denote it by $U\cap U_x$, and by convergence almost all $x_n$'s are in that neighbourhood, $h(U\cap U_x)$ is connected and almost all $h(x_n)$ are in this image. But we can take $U_x$ to have a diameter arbitralily small. Therefore, $h(x_{n_k})$ and $h(x_{n_l})$ have to converge to the same limit. Hence $y=\tilde{y}$. Let us also prove that if every convergent subsequence of a bounded sequence $y_n$ converges to the same limit, then the sequence itself is convergent. Let $y$ be a common limit of all convergent subsequences and suppose that $y_n$is not convergent. Then for every $\varepsilon>0$ we can find a subsequence $y_{n_j}$ such that $d(y,y_{n_j})\ge\varepsilon$. However, since $y_{n_j}$ is also bounded, we can take its convergent subsequence and its limit is different than $y$. But it is also a convergent subsequence of $y_n$. We reached a contradiction. Therefore, $y_n$ is convergent.
For our extension to be properly defined, we need to prove that if $z_n$ is a different sequence converging to $x$, then $h(z_n)$ also converges to $y$. If we assume that $h(z_n)$ converges to $z\neq y$, we can take disjoint neighbourhoods of $z$ and $y$ and intersect them with $h(U)$. In both of these disjoint sets there are almost all $h(z_n)$ and $h(x_n)$ respectively. Now take their counterimages under $h$. They are also disjoint, but $x_n$ and $z_n$ both converge to $x$, so $x$ belongs to both of their boundaries. Again, using existence of a connected neighbourhood of $x$, we can prove that $z=y$. 
We proved that our extension is well defined and because of its construction it is continuous. We will denote it as $\bar{h}$. Next we want to prove that this extension is actually a homeomorphism.

Suppose that there is $y\in\partial\overline{h(U)}$ that is not the image of a point from $\partial U$. Then $\bar{h}(\overline{U})\neq\overline{h(U)}$. However, because $U$ is precompact, then $\overline{U}$ is compact and therefore $\bar{h}(\overline{U})$ is compact. Moreover, $\bar{h}(\overline{U})$ contains $h(U)$. But by definition $\overline{h(U)}$ is the smallest closed set containing $h(U)$. Hence $\overline{h(U)}\subset\bar{h}(\overline{U})$. The inverse inclusion is assured because of the definition of $\bar{h}$. Therefore, we proved that $\bar{h}$ is onto.

We would like to know that a point $x\in\partial U$ is mapped to a point in $\partial h(U)$. Suppose that there is $y\in h(U)$ such that $x\in\bar{h}^{-1}(y)$. We also know that there is some $\tilde{x}\in U$ that is also a counterimage of $y$. We can find disjoint neighbourhoods of $x$ and $\tilde{x}$. However, on $U$ we know that $h$ is a homeomorphism, so the images of these neighbourhoods would also have to be disjoint, but we assumed they have a common point $y$. Hence a point in $\partial U$ is mapped to a point in $\partial h(U)$.

Take $x, \tilde{x}\in\partial U$ and suppose $\bar{h}(x)=\bar{h}(\tilde{x})$. Take sequences $x_n$ and $\widetilde{x_n}$ converging to $x$ and $\tilde{x}$ respectively. We can take disjoint neighbourhoods  $V$ of $x$ and $\tilde{V}$ of $\tilde{x}$ such that almost all $x_n$ and $\widetilde{x_n}$ are elements of $V$ and $\tilde{V}$ respectively. Intersect $V$ and $\tilde{V}$ with $U$. Because $h$ is a homeomorphism the images of these intersections are also disjoint. But it means that $\bar{h}(x)\neq\bar{h}(\tilde{x})$. Hence $\bar{h}$ is one-to-one. 

We know that $\bar{h}$ is a continuous bijection, but $\overline{U}$ is compact and hence $\bar{h}$ is a homeomorphism. Since $\overline{U}$ is compact, $\bar{h}$ is Lipschitz. Similarly $\bar{h}^{-1}$ is Lipschitz. Therefore, $\bar{h}$ is bi-Lipschitz and preserves bounded Lipschitz sets. Then $h$ is also bi-Lipschitz and preserves bounded Lipschitz sets.

\bibliographystyle{plain}

\end{document}